\UseRawInputEncoding
\documentclass[12pt]{amsart}
\usepackage{epsfig, color, amsmath, esint, hyperref, mathrsfs, xcolor, bm, enumitem, mathtools, comment, amsfonts, amssymb, graphicx, psfrag}

\hypersetup{colorlinks}
\hypersetup{citecolor=blue}
\hypersetup{urlcolor=blue}
\makeatother

\theoremstyle{definition}
\def\fnum{equation} 
\newtheorem{Thm}[\fnum]{Theorem}

\newtheorem{Lem}[\fnum]{Lemma}
\newtheorem{Con}[\fnum]{Conjecture}
\newtheorem{Exa}[\fnum]{Example}
\newtheorem{Rem}[\fnum]{Remark}
\newtheorem{Pro}[\fnum]{Proposition}
\newtheorem{Def}[\fnum]{Definition}
\renewcommand{\rm}{\normalshape} 
\newcommand{\diam}{{\text {diam}}}

\newcommand{\iso}{\text{Iso}}
\newcommand{\vol}{\text{vol}}
\newcommand{\arcct}{\text{arccot}}
\newcommand{\vv}{\mathrm{v}}

\newcommand{\Ric}{\text{Ric}}

\newcommand{\kker}{\text{Ker}}

\title{Torus covers with controlled volume and diameter}
\author{Sergio Zamora}\address{Oregon State University}

\begin{document}

\begin{abstract}
    We show that under a lower Ricci curvature bound and an upper diameter bound, a torus admits a finite-sheeted covering space with volume bounded from below and diameter bounded from above. This partially recovers a result of Kloeckner and Sabourau, whose original proof contains a serious gap that currently lacks a resolution. 
\end{abstract}

\maketitle
	
	\section{Introduction}
        Riemannian manifolds with sectional and Ricci curvature lower bounds satisfy multiple analytic, geometric, and topological properties. A common technique to exploit these properties is to study the corresponding universal covers or other suitable covering spaces (see for example \cite{ fukaya-yamaguchi, kpt, kw, naber-zhang}). This technique is particularly useful when these covering spaces are non-collapsed; that is, the volumes of their unit balls have a positive lower bound (see for example \cite{hkrw, huang,  mrw, pan-wang}). The main result of this paper establishes that tori always have non-collapsed covers of controlled diameter.

	\begin{Thm}\label{thm:torus}
		For each $n \in \mathbb{N}$,  $D > 0 $, there are $\varepsilon (n) > 0 $ and $D'(n, D)  > 0$ such that if $M$ is a Riemannian manifold homeomorphic to the $n$-dimensional torus and satisfies  
		\[ \Ric (M) \geq - (n-1),  \hspace{2cm}  \diam (M) \leq D,\] 
		then it admits a covering space $M ' \to M$ that when equipped with the lifted metric satisfies  
		\[     \vol (M') \geq \varepsilon , \hspace{2cm} \diam (M') \leq D' .       \]
	\end{Thm}

	\subsection{Torus stability} The main motivation for Theorem \ref{thm:torus} is  to correct an error in the literature regarding 
    the following result  \cite[Theorem 1.1]{brue-naber-semola}. 
	\begin{Thm}[Bru\'e--Naber--Semola]\label{thm:bns}
		Let $M_i$ be a sequence of Riemannian manifolds homeomorphic to the $n$-dimensional torus and satisfying
		\[ \sec (M_i) \geq -1,  \hspace{2cm}  \diam (M_i) \leq D.\]
        If the sequence $M_i$ converges in the Gromov--Hausdorff sense to a metric space $X$, then  $X$ is homeomorphic to the $m$-dimensional torus for some $m \in \{ 0 , \ldots , n \}$.  
	\end{Thm}

         The proof of Theorem \ref{thm:bns} in \cite{brue-naber-semola}  relies on \cite[Theorem 4.2]{kloeckner-sabourau}, but the proof of \cite[Theorem 4.2]{kloeckner-sabourau} is incorrect and it is unclear whether it can be repaired.    In \cite{brue-naber-semola}, the authors only used \cite[Theorem 4.2]{kloeckner-sabourau} to deduce the statement of Theorem \ref{thm:torus}.  With Theorem \ref{thm:torus} established, the proof of Theorem \ref{thm:bns} given in \cite{brue-naber-semola} now succeeds.

           We now outline the gap in the proof of \cite[Theorem 4.2]{kloeckner-sabourau}\footnote{This was originally pointed out by Cameron Rudd.}
           .   The main tool that is used in said proof is \cite[Proposition 3.4]{kloeckner-sabourau},  which in turn uses the construction  \cite[Definition 2.5]{kloeckner-sabourau}. Given $\ell \in \mathbb{N}$ and a primitive cohomology class $\lambda \in H^1 (M ; \mathbb{Z})$ in a torus $M$, a finite-sheeted cover $\hat{M}_{\lambda , \ell }$ is constructed and claimed to be the $\ell$-sheeted cover associated to $\lambda$. This claim relies on \cite[Lemma 2.4]{kloeckner-sabourau}, which states that a certain $(n-1)$-cycle $H_{\lambda} $ in $M$ represents the Poincar\'e dual of $\lambda$. This is false. The cycle $H_{\lambda}$ is constructed using only the information encoded in the $2$-sheeted cover $\overline{M}_{\lambda}$ associated to $\lambda$, but distinct primitive cohomology classes $\lambda ,  \lambda ' \in H^1(M ; \mathbb{Z})$ can have the same associated $2$-sheeted cover $ \overline{M}_{\lambda } = \overline{M}_{\lambda '}$. 
        
        \color{black}

        In \cite{zhou},  Zhou showed that in dimension $n = 4$, Theorem \ref{thm:bns} holds if one replaces the lower sectional curvature bound by a two-sided Ricci curvature bound. Their argument again invokes \cite[Theorem 4.2]{kloeckner-sabourau}, but Theorem \ref{thm:torus}  can be used in its place.

        \subsection{Proof strategy} Theorem \ref{thm:torus} will be derived from Theorem \ref{thm:separated-subgroup} below, together with  a result from \cite{guth} regarding volumes of balls in univeral covers of aspherical manifolds (see Theorem \ref{thm:guth}).

	\begin{Thm}\label{thm:separated-subgroup}
		Let $\mathcal{M}$ be a class of pointed proper simply-connected geodesic spaces that is pre-compact in the pointed Gromov--Hausdorff topology. For each $n \in \mathbb{N}$ and $D > 0$ there is $ D' ( n,  D , \mathcal{M})>0$ such that if $(X,p) \in \mathcal{M}$, and  a discrete group $ G \leq \iso (X)$  satisfies  
		\[ G \cong \mathbb{Z}^n, \hspace{2cm}  \diam (X / G ) \leq D  , \] then there is a finite-index subgroup $\Gamma  \leq G$ with $\diam (X / \Gamma  )  \leq D'  $ and
		\[  d ( g p , p  ) \geq D \text{ for all } g \in \Gamma  \backslash \{ e \} .  \]
	\end{Thm}

	Theorem \ref{thm:separated-subgroup} will be obtained by combining two results. The first ingredient is a classical construction from \cite[Section 5C]{gromov} (see also \cite[Lemma 3.1]{colding}). Recall that for a proper geodesic space $X$ and a discrete group $G \leq \iso (X)$, the abelianization $G^{ab} : = G/ [G,G]$ admits a natural action on $X / [G,G]$.
	
	\begin{Thm}[Gromov]\label{thm:gromov}
		Let $X$ be a proper geodesic space and $ G \leq \iso (X)$ a discrete group with $  \diam (X / G ) \leq D $.  Then for any $p \in X / [G,G]$, there is a set  $\{ g_1, \ldots , g_n  \} \subset G ^{ab}$ with
		\begin{gather*}
			d ( g p, p )  \geq D  \text{ for all } g \in \langle g _1, \ldots , g _n \rangle   \backslash \{ e \}          , \\
			d ( g _j p, p )  \leq 2D \text{ for all } j \in \{1 , \ldots , n \} ,      
		\end{gather*}  
		and such that  $  \{ g_1 , \ldots , g _n  \} $ is a basis of $ G^{ab} \otimes \mathbb{R}$. 
	\end{Thm}

	The second ingredient is the main technical result of this paper. 
	
	\begin{Thm}\label{thm:main-technical}
		Let $(X_i,p_i)$ be a sequence of simply-connected pointed proper geodesic spaces and $\Gamma _ i \leq G_i \leq \iso (X_i)$ discrete groups with $      \Gamma _i \cong  G_i \cong \mathbb{Z}^n $ and $  \diam (X_i / G_i ) \leq D . $ Assume 
		\[       d ( g  p_i, p_i) \geq D \text{ for all }  g \in \Gamma _i \backslash \{ e \} ,                  \]
		and that for each $i \in \mathbb{N} $ there  is an isomorphism $\varphi _i : \mathbb{Z}^n \to \Gamma _i  $ with 
		\begin{equation}\label{eq:phi-no-escape}
			d (\varphi _i (e_j) p_i,p_i) \leq 2D  \text{ for all } j \in \{1, \ldots , n \}.        
		\end{equation}
		If the sequence $(X_i,p_i)$ converges in the pointed Gromov--Hausdorff sense to a space $(X,p)$, then 
		\begin{equation}\label{eq:diam-conclusion}
			\sup _i    \diam (X_i / \Gamma _ i )     < \infty .    
		\end{equation}
	\end{Thm}

       The following result implies that Theorem \ref{thm:torus} cannot be generalized to nilmanifolds or other families of manifolds admitting metrics of almost non-negative Ricci curvature and having torsion-free non-virtually abelian fundamental groups   \cite[Proposition 1.5]{zamora-zhu}.

	\begin{Pro}\label{pro:zhu}
		Let $M_i$ be a sequence of closed $n$-dimensional Riemannian manifolds with torsion-free fundamental group and 
		\[\Ric (M_i) \geq -\frac{1}{i}, \hspace{1.7cm} \diam (M_i ) \leq D . \hspace{0.3cm}  \] 
		If there is a sequence of covering spaces $M_i ' \to M_i$ with    
		\[    \vol (M_i') \geq \varepsilon  , \hspace{2cm}   \diam (M_i ') \leq D'  \]
		for some $\varepsilon > 0 $ and $D' > 0 $, then $\pi_1(M_i)$ is virtually abelian for  $i$ large enough.
	\end{Pro}

	We point out that all the tools we use throughout this paper are elementary with the exception of Theorem \ref{thm:guth}, which is used only to derive Theorem \ref{thm:torus} from Theorem \ref{thm:separated-subgroup}.

    \subsection{Outline and contents}\label{sec:outline}	Theorem \ref{thm:main-technical} is proven by contradiction. If \eqref{eq:diam-conclusion} fails, then after passing to a subsequence, the groups $G_i$ converge to a group $G\leq \iso (X)$  with $X/G$ compact,  and the groups  $\Gamma _i $ converge to a discrete group $\Gamma \leq \iso (X) $  with 
	\[  \hspace{1.5cm}  \mathbb{Z}^n \leq \Gamma, \hspace{3cm}   X/ \Gamma \text{ not compact.} \] 
	In that case, the group $G$ contains a discrete co-compact subgroup $H \leq G$ isomorphic to $\mathbb{Z}^N$ with $N > n$. This implies that for $i$ sufficiently large, $\mathbb{Z}^N$ admits a discrete co-compact action on $X_i$ that resembles the $\mathbb{Z}^N$-action on $X$.  Note that since $N > n$, the action of $\mathbb{Z}^N$ on $X_i$ is very much not faithful. 
	
    Here we apply Theorem \ref{thm:monodromy}; a monodromy-like construction of covering spaces similar to the ones from  \cite{fukaya-yamaguchi, pan-wang, zamora}. From the unfaithfulness of the $\mathbb{Z}^N$-action on $X_i$, this construction produces a non-trivial covering space of $X_i$, contradicting the fact that it is simply-connected.

   In Section \ref{sec:lit} we discuss related results.  In Section \ref{sec:prelim} we recall the results required for Theorem  \ref{thm:main-technical}.  In Section \ref{sec:covers} we prove Theorem \ref{thm:monodromy}. In Section \ref{sec:proof} we prove Theorem \ref{thm:main-technical} following the above outline, and deduce Theorem \ref{thm:separated-subgroup} from Theorems \ref{thm:gromov} and \ref{thm:main-technical}. In Section \ref{sec:ricci} we prove Theorem \ref{thm:torus} combining Theorem  \ref{thm:separated-subgroup} with \cite[Corollary 3]{guth}.

    \section{Related results and problems}\label{sec:lit}
    
    It would be interesting to know if in Theorem \ref{thm:torus} one can drop the Ricci curvature hypothesis. This would be true if one could remove the pre-compactness condition on $\mathcal{M}$ in Theorem \ref{thm:separated-subgroup}.   
	
	\begin{Con}\label{con:general}
	    For each $n \in \mathbb{N}$, $D>0$,  there is $ D'(n, D)>0$ such that if $(X, p)$ is a pointed simply-connected proper geodesic space, and  $ G \leq \iso (X)$ a discrete group with 
		\[ G \cong \mathbb{Z}^n, \hspace{2cm}  \diam (X / G ) \leq D  , \] 
        then there is a finite-index subgroup $\Gamma  \leq G$ with $\diam (X / \Gamma  )  \leq D'  $ and
		\[  d ( g p , p  ) \geq D \text{ for all } g \in \Gamma  \backslash \{ e \} .  \]
	\end{Con}

    When a proper geodesic space $X$ admits a discrete co-compact group of isometries $G\leq \iso (X)$ isomorphic to $\mathbb{Z}^n$, the \emph{stable norm} on $G $ is defined as 
        \[      \Vert g \Vert _{st} : = \lim_{k \to \infty } \frac{d( g ^k  p , p ) }{ k }                             \]
    for some $p \in X$.    It is well known that once an isomorphism  $G \cong \mathbb{Z}^n$ has been chosen,  $\Vert \cdot \Vert _{st}$ defines a norm on $\mathbb{R}^n$ and that  $X$ is at finite Gromov--Hausdorff distance from $( \mathbb{R}^n , \Vert \cdot  \Vert _{st} )$ (see \cite{burago, burago-burago-ivanov}). Moreover, in \cite[Theorem 1.1]{cerocchi-sambusetti} it is shown that this distance can be bounded in terms of $n$, $\diam (X/ G )$, and the asymptotic volume defined as
    \[     \omega (G, d) : = \lim_{r \to \infty} \frac{\# \{g \in  G \vert d(gp,p) < r \} }{r^n}  .            \]
    \begin{Thm}[Cerocchi--Sambusetti]\label{thm:cs-volume}
 Let  $(X, p)$ be a pointed proper geodesic space, and  $ G \leq \iso (X)$ a discrete group with 
		\[ G \cong \mathbb{Z}^n, \hspace{2cm}  \diam (X / G ) \leq D  , \hspace{2cm} \omega (G,d) \leq \Omega.  \] 
        Then for all $g \in \mathbb{Z}^n$ one has
        \begin{equation}\label{eq:cs}
            \vert d(gp,p) - \Vert g \Vert _{st} \vert \leq C(n, D, \Omega) . 
        \end{equation}
\end{Thm}
    As a consequence of Theorem \ref{thm:cs-volume}, we obtain the following analogue of Theorem \ref{thm:separated-subgroup}. This implies that  Conjecture \ref{con:general} holds when  the simple-connectedness hypothesis is replaced by a bound on the asymptotic volume.  \color{black}
\begin{Thm}
 For each $n \in \mathbb{N}$, $D>0 $, $\Omega > 0$,  there is $ D'(n, D, \Omega )>0$ such that if $(X, p)$ is a pointed proper geodesic space, and  $ G \leq \iso (X)$ a discrete group with 
		\[ G \cong \mathbb{Z}^n, \hspace{2cm}  \diam (X / G ) \leq D  , \hspace{2cm} \omega (G,d) \leq \Omega,  \] 
        then there is a finite-index subgroup $\Gamma  \leq G$ with $\diam (X / \Gamma  )  \leq D'  $ and
		\[  d ( g p , p  ) \geq D \text{ for all } g \in \Gamma  \backslash \{ e \} .  \]    
\end{Thm}

    \begin{proof}
        Pick $p \in X$ and identify $G$ with $\mathbb{Z}^n$.  
        Since the set 
        \[   S : =   \{ g \in \mathbb{Z}^n \vert d (gp,p) \leq 2D \}               \]
        generates $\mathbb{Z}^n$, it also spans $\mathbb{R}^n$. By Theorem \ref{thm:cs-volume}, there is $C(n, D, \Omega) > 0 $ for which \eqref{eq:cs}  holds for all $g \in \mathbb{Z}^n$. In particular, 
        \[     \Vert g \Vert _{st} \leq 2D + C                    \]
        for each $g \in S$.  Hence for any $\vv \in \mathbb{R}^n$, there is $g \in \mathbb{Z}^n$ with $\Vert \vv - g \Vert_{st} \leq n (2D + C )  $.         By John's Theorem \cite[Theorem 3.3]{milman-schechtman}, there is a linear isomorphism $\alpha : \mathbb{R}^n \to \mathbb{R}^n$ with 
        \[    \frac{\Vert \vv \Vert _2}{\sqrt{n}} \leq \Vert \alpha \vv \Vert _{st} \leq \Vert \vv \Vert _2      \]
        for all $\vv \in \mathbb{R}^n$. Pick $\{ \vv_1 , \ldots , \vv_n \} \subset \mathbb{R}^n$ an $\Vert \cdot \Vert _2$-orthogonal basis with  $\Vert \vv _i \Vert_2 = M$ for each $i$, with $M > 0 $ to be chosen later.  Then there are $g_1, \ldots , g_n \in \mathbb{Z}^n$ with 
        \begin{equation}\label{eq:good-gi}
        \Vert g_i - \alpha \vv_i \Vert _{st} \leq n (2D + C)    
        \end{equation}
        for each $i$. A simple calculation shows that if $M  =  2 n^3 (2D + C) $, then the set $\{ g_1, \ldots , g_n \}$ is linearly independent, and for any $g \in \langle g_1, \ldots, g_n \rangle \backslash \{ e \}$ one has 
        \[    \Vert \alpha ^{-1}  g \Vert _2 \geq  D + C .                      \]        
        Set $\Gamma : = \langle g_1, \ldots , g_n \rangle $. Then for any $g \in \Gamma \backslash \{ e \}$ one has 
        \[            d( gp,p)  \geq  \Vert g \Vert _{st} - C  \geq    \Vert \alpha ^{-1} g \Vert_2  - C   \geq  D .
        \]
        On the other hand, by \eqref{eq:good-gi} one has 
        \[  \Vert g_i \Vert _{st} \leq \Vert \alpha \vv_i \Vert _{st} + n (2D + C) \leq  M + n (2D + C)          \]
        for each $i$, so for any  $\vv \in \mathbb{R}^n$, there is $g \in \Gamma$ with 
        \[    \Vert \vv - g \Vert _{st} \leq Mn + n^2 (2D + C) .                  \]
        For any $x \in X$, there is $\vv \in \mathbb{Z}^n $ with $d(\vv p , x) \leq D$, and by the above observation,  $g \in \Gamma $ with 
        \begin{eqnarray*}
            d(gp,x) & \leq & d (gp,\vv p) + D \\
            & \leq & \Vert \vv - g \Vert _{st} + C + D \\
            & \leq & Mn + (n ^2 + 1 ) (2D + C). 
        \end{eqnarray*}
         Therefore,  if $D' (n, D, \Omega) : = 2 Mn + 2(n ^2 + 1 ) (2D + C)$, then 
        \[    \diam (X/ \Gamma) \leq D'.           \]
        \end{proof}

    When $M$ has almost non-negative Ricci curvature, Theorem \ref{thm:torus} follows from the work of Colding  \cite{colding}. However, their approach is substantially different as they first show that the universal cover $\tilde{M}$ is everywhere almost Euclidean \cite[Lemma 3.5]{colding}, a condition that does not hold under an arbitrary Ricci curvature lower bound. Note that they work under the weaker hypothesis that $M$ has first Betti number equal to $n$, rather than being homeomorphic to the $n$-dimensional torus. Nevertheless, they ultimately show that in the presence of almost non-negative Ricci curvature,  for $n \neq 3$, these two hypotheses are equivalent \cite[Theorem 0.2]{colding} (in dimension $n=3$, this equivalence was later obtained as a consequence of  Perelman's solution to the Geometrization Conjecture).

        Recently, Rong generalized Theorem \ref{thm:bns} to nilmanifolds \cite{rong} with a different proof strategy that does not rely on Theorem \ref{thm:torus}. Certainly, Proposition \ref{pro:zhu} shows that a result analogous to Theorem \ref{thm:torus}  would not hold for nilmanifolds.

	\section{Preliminaries}\label{sec:prelim}

	\subsection{Group theory} Given any set $S$, we denote by $F_S$ the free group with letters the elements of $S$. In a group $G$, we say a set $S \subset G$ is \textit{symmetric} if it contains the identity and for all $s \in S $ one has $s^{-1} \in S$.
	\begin{Def}
		Let $G$ be a group and $S \subset G$ a symmetric generating set. We say $S$ is a  \textit{defining set} if $G$ is presented with $S$ as generating set and $\{ abc^{-1} \in F_S  \, \vert  \, a, b, c \in S , ab= c \} $ as set of relations. 
	\end{Def}
	
	A given group can admit several different-looking presentations, but the ones produced using defining sets are  particularly nice.  It turns out that finitely presented groups always admit finite defining sets \cite[Theorem 7.A.8]{cornulier}.
	
	\begin{Thm}\label{thm:defining}
		Let $G$ be a group and $S\subset G$ a symmetric generating set. If $G $ admits a presentation $G  = \langle S \vert R \rangle$ with $R$ consisting of words of length $\leq \ell$, then $ S^{\lfloor \frac{\ell + 2 }{3}   \rfloor }$ is a defining set. 
	\end{Thm}
	
	The following is part of the well known Milnor--\v{S}varc Lemma (see \cite{de-la-harpe} for example).

	\begin{Lem}\label{lem:generators}
		Let $(X,p)$ be a pointed proper geodesic space and $G \leq \iso (X)$ a co-compact group of isometries. If $R > 2 \cdot  \diam (X/ G)$, then 
		\[     S : = \{ g \in G \vert   d (gp,p) < R  \}                 \]
		is a symmetric pre-compact generating set of $G$. In particular, $G$ is compactly generated.
	\end{Lem}
	
	The following result is due to Pontrjagin \cite[Section 35]{pontrjagin} (see also \cite[Lemma 2.4.2]{rudin}). 
	
	\begin{Lem}\label{lem:lattice}
		Any locally compact compactly generated Hausdorff abelian group admits a discrete co-compact subgroup isomorphic to $\mathbb{Z}^m$ for some $m \in \mathbb{N}$. 
	\end{Lem}

	\subsection{Gromov--Hausdorff convergence} 
	
	We assume the reader is familiar with pointed and equivariant Gromov--Hausdorff convergence, introduced in \cite[Section 6]{gromov-polynomial} and \cite[Chapter 1]{fukaya}, respectively. 
	One of the main features of this framework is that once one has pointed Gromov--Hausdorff convergence, after passing to a subsequence one can promote it to  equivariant Gromov--Hausdorff convergence  \cite[Proposition 3.6]{fukaya-yamaguchi}.
	
	\begin{Thm}[Fukaya--Yamaguchi]\label{thm:fukaya-yamaguchi}
		Let  $(X_i,p_i )$ be a sequence of pointed proper geodesic spaces that converges in the pointed Gromov--Hausdorff sense to a space $(X,p)$, and $G_i \leq \iso (X_i)$ a sequence of closed groups of isometries. Then after passing to a subsequence, there is $G \leq \iso (X)$ such that the triples $(X_i,p_i,G_i)$ converge to the triple $(X,p,G)$ in the equivariant Gromov--Hausdorff sense. 
	\end{Thm}

	Another feature of  equivariant Gromov--Hausdorff convergence is  its compatibility with taking quotients. This follows directly from \cite[Theorem 2.1]{fukaya}.

	\begin{Thm}[Fukaya]\label{thm:fukaya}
		Let $X_i$ be a sequence of proper geodesic spaces, $p_i \in X_i$, and $G_i \leq \iso (X_i)$. If the sequence of triples     $(X_i,p_i, G_i)$ converges in the equivariant Gromov--Hausdorff sense to a triple $(X,p,G)$, then the sequence of quotients  $(X_i / G_i , [p_i])$ converges in the pointed Gromov--Hausdorff sense to $(X/G, [p])$.  
	\end{Thm}

	\section{Construction of covering spaces}\label{sec:covers}
	
	In this section, we present a construction of covering spaces out of group actions, likely of independent interest.  The novelty of this result is that it allows one to take different choices of $B$, $S$, and $T$, rather than being pre-determined by the action (cf. \cite[Appendix A]{fukaya-yamaguchi}, \cite[Section 5]{pan-wang}, \cite[Section 2.5]{zamora}). This flexibility is required for its application to Theorem \ref{thm:main-technical}.

	\begin{Thm}\label{thm:monodromy}
		Let $X$ be a proper geodesic space, $\Gamma  \leq \iso (X)$ a closed group,  $B \subset X$ a connected open set with $\Gamma B = X $,  $S \subset   \Gamma$ a symmetric generating set, and $T := S^M$ for some $M \geq 1$. Assume the following holds:
		\begin{enumerate}[label= \textbf{\roman*}]
			\item For all $s \in S$ we have $sB \cap B \neq \emptyset$.\label{item:c1}
            \item If $t \in T^6$ satisfies $tB \cap B \neq \emptyset$, then $t \in T$. \label{item:c2}
			\item If $g \in \Gamma$ satisfies $gB \cap B \neq \emptyset$, then $gB \subset TB$.\label{item:c3}
		\end{enumerate}
		Let $\tilde{\Gamma}$ be the abstract group generated by $T^3$ with relations 
		\begin{center}
			$ab = c$ in $\tilde{\Gamma}$ whenever $a,b,c \in T^3$, and $ab = c $ in $\Gamma$.
		\end{center}
		If we denote the  canonical embedding $T^3  \hookrightarrow \tilde{\Gamma} $ as $t \mapsto t^{\sharp } $, then there is a unique group morphism $\Phi :  \tilde{\Gamma } \to \Gamma$ that satisfies $\Phi (t^{\sharp }) = t$ for all $t \in T^3$. Equip $ \tilde{\Gamma } $ with the discrete topology, and consider the topological space
		\begin{center}
			$\tilde{X} := \left(  \tilde{\Gamma }  \times B \right) / \sim$,
		\end{center}
		with $\sim$ defined as
        \begin{equation}\label{eq:sim}
		(gt^{\sharp}, x) \sim (g,t x) \, \text{ whenever }\, g \in \tilde{\Gamma}, t \in T,  \text{ and }\,  x ,  tx \in B .
        \end{equation}
		Then
                \begin{itemize}
                    \item $\tilde{X}$ is connected.
                    \item The map $\Psi : \tilde{X} \to X$ given by
		\[ \Psi  (g, x) : = \Phi (g) x    \]
		is a covering map. 
                    \item $\Psi$ is a homeomorphism if and only if 
		\begin{equation}\label{eq:psi-homeo}
			\{ g \in \tilde{ \Gamma }  \, \vert \, \Phi (g) B \cap B \neq \emptyset \} \subset T^{\sharp} . 
		\end{equation}   
                \end{itemize} 
	\end{Thm}

	\begin{Exa}\label{exa:lie}
		Let $\Gamma = X$ be a connected Lie group equipped with a left-invariant metric, $B \subset X$ the exponential of a small ball in the Lie algebra, and 
		\begin{equation}\label{eq:good-s}
			S = T = \{ g \in \Gamma \vert gB \cap B \neq \emptyset \}. 
		\end{equation}
		Then conditions \eqref{item:c1}, \eqref{item:c2}, \eqref{item:c3} follow trivially, and $\tilde{\Gamma} = \tilde{X}$ is the universal cover of $\Gamma$. 
	\end{Exa}

	\begin{Exa}\label{exa:general-good-s}
		Let $X$ be a proper geodesic space, $\Gamma  \leq \iso (X) $ a closed group acting co-compactly with $\diam (X /\Gamma ) = D$, $B \subset X$ an open ball of radius $> D$, and $S = T $ be as in \eqref{eq:good-s}. Then all conditions \eqref{item:c1}, \eqref{item:c2}, \eqref{item:c3} follow trivially and $\Psi : \tilde{X} \to X$ is a regular cover  with Galois group $\kker (\Phi)$ (see \cite[Section 2.5]{zamora}).
	\end{Exa}

	\begin{Exa}\label{ex:irrational-rotation}
		Let $r \in (0,  1/10 ) $ be an irrational number, $X = \mathbb{S}^1 \times \mathbb{R}$,  and 
		\[ B =  \{ (e ^{is} , t ) \in X \, \vert \, s \in ( - r, r ) , t \in \mathbb{R} \}  .   \]
		Set $\Gamma  = \mathbb{Z}$, $S = \{ -1, 0, 1 \}$, $T = \{ -2, -1, 0, 1,2 \}$, with the action given by 
		\[   k \cdot (z,t)  =  \left(  e^{irk} z \, , \,  t   \right)  \text{ for } k \in \Gamma  , z \in \mathbb{S}^1 , t \in \mathbb{R}.                      \]
		This setting satisfies the conditions of the theorem. Then $\tilde{\Gamma} = \mathbb{Z}$,  $\tilde{X} = \mathbb{R}^2$, and $\Psi : \tilde{X} \to X$ is the standard covering map. 
        
        On the other hand, if one were to take 
		\[   B  =  \{    (e ^{is} , t ) \in X \, \vert \, s \in (  - r  \cdot \arcct (t)  ,  \,  r  \cdot  \arcct (t)   ) , \,t \in \mathbb{R}  \}                  \]
		with $\arcct  : \mathbb{R} \to (0, \pi) $, then conditions \eqref{item:c1} and \eqref{item:c2} would hold, but condition \eqref{item:c3} would fail, and $\tilde{X}$ would not be a covering space of $X$ (see Figure \ref{fig:strips}).

		\begin{figure}[h!]
			\psfrag{a}{$B$}
			\psfrag{b}{$\tilde{X}$}
			\includegraphics[scale = 0.5]{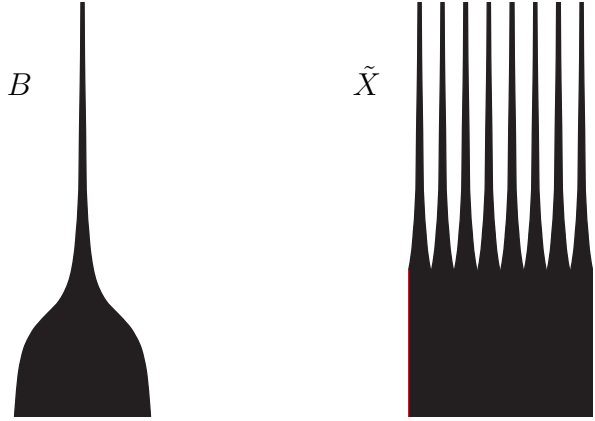}
			\caption{$\tilde{X}$ is obtained from copies of $B$ glued one after the other.}\label{fig:strips}
		\end{figure}

	\end{Exa}

	The proof of Theorem \ref{thm:monodromy} is obtained from a series of lemmas.  

	\begin{Lem}\label{lem:relation-is-good}
	The relation $\sim$  given by \eqref{eq:sim} is an equivalence relation.
	\end{Lem}
	
	\begin{proof}
            Let $(a,x), (b,y ) \in \tilde{\Gamma} \times B$ and assume $(a,x ) \sim (b,y)$. This means there is $t \in T$ with $a= bt^{\sharp}$, $y=tx$. Then $b = a (t^{-1})^{\sharp}$ and $x = t^{-1}y$, so $(b,y) \sim  (a,x)$ and  $\sim$ is symmetric. 
            
            To see that $\sim$ is transitive, assume $(a,x), (b,y ), (c,z) \in \tilde{\Gamma} \times B$ satisfy $(a,x) \sim (b,y) \sim (c,z)$. This means
            \[     a = b t_1 ^{\sharp} = c t_2 ^{\sharp } t_1 ^{\sharp} , \, z = t_2y = t_2 t_1  x                          \]
            for some $t_1, t_2 \in T$. By \eqref{item:c2}, we have $t_2t_1 \in T$, so $(t_2t_1)^{\sharp} = t_2 ^{\sharp}t_1 ^{\sharp}$, and 
            \[ a = c (t_2t_1)^{\sharp},\,  z = (t_2t_1)x ,\] 
            meaning that $(a,x) \sim (c,z)$. 
	\end{proof}

	\begin{Lem}\label{lem:tilde-conn}
		$\tilde{X}$ is connected.
	\end{Lem}
	\begin{proof}
		Since $S^j \subset T^3$ for all $j \in \{ 1, \ldots , 3M \}$, we have $ (S^{j-1})^{\sharp} S^{\sharp} = (S^j )^{\sharp} $ for such $j$'s. Using this, one gets $(S^{\sharp} )^{3M} = (T^3)^{\sharp}$, meaning  $S^{\sharp}$ generates $(T^3)^{\sharp} $, and consequently $\tilde{\Gamma}$. The result then follows from \eqref{item:c1} and the fact that $B$ is connected.
	\end{proof}

	Fix $p \in X$ and define 
	\[  \Sigma :  = \{ g \in \tilde{\Gamma} \, \vert \,  p \in \Phi (g)  B   \} . \]
	Let $K \subset \Sigma $ be a maximal subset containing the identity and such that 
	\begin{center}
		if $a,b \in K$ are distinct, then $b \notin aT^{\sharp}$.   
	\end{center}

	\begin{Lem}\label{lem:k-separated}
		If $a ,b \in K$ are such that $a (T^3)^{\sharp} \cap  b (T^3)^{\sharp} \neq \emptyset$, then $a = b$. 
	\end{Lem}
	\begin{proof}
		Assume there are $t_1 , t_2 \in T^3$ with $at_1 ^{\sharp} = bt_2 ^{\sharp}$. Then $\Phi (b)^{-1} p \in B$ and
		\[    t_1 t_2 ^{-1} \Phi (b)^{-1}p = \Phi (a)^{-1} p  \in B  ,    \]
		so from \eqref{item:c2} we deduce $t_1 t_2^{-1} \in T$.  This implies $(t_1 t_2 ^{-1})^{\sharp} = t_1 ^{\sharp}  (t_2^{-1})^{\sharp}$, and 
		\[ b = a t_1^{\sharp} (t_2 ^{-1})^{\sharp} = a (t_1 t_2 ^{-1})^{\sharp} \in  aT ^{\sharp} . \]
		Consequently, $a= b$. 
	\end{proof}

	\begin{Lem}\label{lem:k-big}
		For each $g \in \Sigma$, there is a unique $a \in K$ such that $g \in a T^{\sharp}  $. 
	\end{Lem}
	
	\begin{proof}
		Existence follows from the maximality of  $K$, and uniqueness from Lemma \ref{lem:k-separated}. 
	\end{proof}
	
	Fix $g_0 \in K$, set  $U : = \Phi (g_0) B$, and define
	\[    \Sigma ' : = \{ g \in \tilde{ \Gamma } \vert \Phi (g) B \cap U \neq \emptyset \}  .    \]

	\begin{Lem}\label{lem:k-big-big}
		For each $g \in \Sigma ' $, there is a unique $a \in K$ such that $g \in a (T^2)^{\sharp}  $. 
	\end{Lem}
	
	\begin{proof}
		By the definition of $U$, we have  $\Phi (g^{-1}g_0) B \cap B \neq \emptyset $, so by \eqref{item:c3} there are $t \in T$ and $x \in B$ with
		\[      t x = \Phi (g^{-1} g_0) \Phi (g_0 ^{-1})p = \Phi (g^{-1}) p .             \]
		This means $gt^{\sharp} \in \Sigma $, so existence follows from Lemma \ref{lem:k-big}. Uniqueness follows from Lemma \ref{lem:k-separated}.
	\end{proof}
	
	\begin{Lem}\label{lem:monodromy-i}
		\rm  For $(g,x ) \in \tilde{\Gamma } \times B$, its class in $\tilde{X}$ is mapped via $\Psi$ into $U$ if and only if 
		\[
		(g,x) \in  \bigcup_{a \in K } \left( \bigcup_{\,\,\,\,t \in T^2} \left( \{ a  t ^{\sharp} \} \times ( ( t^{-1}  \Phi (a)^{-1} U ) \cap B  )    \right)     \right)    .
		\]
		Consequently, the preimage of $U$ in $\tilde{X}$ is given by 
		\[   \Psi ^{-1}(U) = \bigcup_{a \in K } \left( \bigcup_{\,\,\,\,t \in T^2} \left( \{ a  t ^{\sharp} \} \times ( ( t^{-1}  \Phi (a)^{-1} U ) \cap B  )    \right)     \right)   / \sim   .   \]
	\end{Lem}
	
	\begin{proof}
		For $a \in K$, $t \in T^2$, $x \in (t^{-1} \Phi (a)^{-1} U)\cap B $, we directly compute
		\[   \Psi (a t^{\sharp }, x) = \Phi (a) t x \in U .    \]
		On the other hand, if $(g,x) \in \tilde{\Gamma} \times B$ satisfies  $\Phi (g) x \in U$, by Lemma \ref{lem:k-big-big} there are $a \in K$ and $t \in T^2$ with $g = a t^{\sharp}$, so $x \in \Phi (g)^{-1}  U =  t ^{-1} \Phi (a)^{-1} U $.\end{proof}
	
	We say that a subset $A \subset \tilde{\Gamma } \times B$ is \textit{saturated} if it is a union of equivalence classes of the relation $\sim$. 
	\begin{Lem}\label{lem:monodromy-ii}
		\rm For $a \in K$, the sets 
		\[   W_a : =   \bigcup_{\,\,\,\,t \in T^2} \left( \{ a  t ^{\sharp} \} \times ( ( t^{-1}  \Phi (a)^{-1} U ) \cap B  )    \right)     \subset \tilde{\Gamma } \times B    \]
		are open, disjoint, and saturated.
	\end{Lem}
	\begin{proof}
		The fact that they are open is straightforward, since $(t^{-1} \Phi (a)^{-1}U) \cap B$ is open in $B$ for each $t \in T^2$. The fact that they are disjoint follows from Lemma \ref{lem:k-separated}.  To see that they are saturated, assume an element $(g, x) \in \tilde{\Gamma} \times B$ is equivalent to an element $(at_1^{\sharp}, y) \in W_a$.  By Lemma \ref{lem:monodromy-i} we have 
		\[ \Psi (g,x) = \Psi (at_1^{\sharp}, y) \in U.\] 
		Again by Lemma \ref{lem:monodromy-i}, this implies $(g,x) \in W_b$ for some $b \in K$ and $g = b t_2^{\sharp}$ for some $t_2 \in T^2$.        Since $(g,x) \sim (at_1^{\sharp}, y)  $,  there is $s \in T$ with 
		$at_1^{\sharp} s^{\sharp} = g =  b t_2 ^{\sharp} $. From Lemma \ref{lem:k-separated} we conclude that $a = b$. This shows $(g,x)\in W_a$.
	\end{proof}
	
	\begin{Lem}\label{lem:monodromy-iii}
		\rm For each $a \in K$, the image of $W_a$ in $\tilde{X}$ is sent  homeomorphically via $\Psi$ onto $U$.
	\end{Lem}
	
	\begin{proof}
		To check surjectivity, pick $x \in U$. Since $\Phi (a^{-1} g_0 ) B \cap B \neq \emptyset$, by \eqref{item:c3} there are $t \in T$ and $y \in B$ with 
		\[     ty = \Phi (a^{-1}g_0) \Phi (g_0 )^{-1}x  = \Phi (a )^{-1}  x     ,         \]
		so $\Psi (at^{\sharp}, y) = x$.     To check injectivity, assume 
		\[   \Psi ( a t_1^{\sharp} , x_1 ) = \Psi ( at_2^{\sharp}, x_2)            \]
		for some $t_1, t_2 \in T^2$, $x_1 \in ( t^{-1}_1 \Phi (a)^{-1}  U ) \cap B $, $x_2 \in ( t_2^{-1}\Phi (a)^{-1} U   )\cap B$. Then $t_1 x_1 = t_2 x_2$, and 
		\[  t_2 ^{-1} t_1 x_1 = x_2 .    \]
		From \eqref{item:c2}, we deduce $t_2^{-1}t_1 \in T$, so
		\[   ( a t_2 ^{\sharp} , x_2) \sim (a t_2^{\sharp} (t_2 ^{-1} t_1 )^{\sharp} , t_1 ^{-1} t_2 x_2) = ( a t_1 ^{\sharp} , x_1).                                \]
		To check that $\Psi \vert_{W_a}$ is open, take $\mathcal{O} \subset W_a$ open and saturated containing the class of 
		\[( a t ^{\sharp} ,  x) \in  \{ a t ^{\sharp}\} \times ( (  t^{-1} \Phi (a) ^{-1} U  )  \cap  B)   . \]
		Then $\Psi $ sends $ ( \{ at^{\sharp}\} \times  B    ) \cap \mathcal{O} $ to an open neighborhood of $\Phi( a)t(x)$. Since $( a t ^{\sharp} ,  x) $ was arbitrary, $\Psi \vert _{W_a} $ is open.
	\end{proof}

	\begin{proof}[Proof of Theorem \ref{thm:monodromy}:] 
		By Lemma  \ref{lem:tilde-conn}, $\tilde{X}$ is connected.  By Lemmas \ref{lem:monodromy-i}, \ref{lem:monodromy-ii}, and \ref{lem:monodromy-iii}, $U$ is an evenly covered neighborhood. Since $p$ was arbitrary, $\Psi $ is a covering map. The number of sheets is precisely $\vert K \vert $, so $\Psi$ is a homeomorphism if and only if $K$ consists of a single element, which happens if and only if \eqref{eq:psi-homeo} holds. 
	\end{proof}

	\section{Proof of the diameter theorems}\label{sec:proof}

    In this section, we prove Theorems   \ref{thm:main-technical} and \ref{thm:separated-subgroup} by following the outline presented in Section \ref{sec:outline}. Before we proceed to the proof, we present an example showing that Theorem \ref{thm:main-technical} fails if one removes the simple-connectedness hypothesis.

    \begin{Exa}\label{ex:cylinder}
        Let $X_i : = \mathbb{S}^1_{i} \times \mathbb{R}$, where $\mathbb{S}^1_i $ represents the circle of radius $i$, and let $g_i \in \iso (X_i)$ be defined as 
        \[      g_i  (x,t)  : =   (  e^{2 \pi \sqrt{-1} / i } \cdot x , t + 1/ i  )  .   \]
        Set $G_i : = \langle g_i \rangle $ and $ \Gamma _i : = \langle g_i ^{i} \rangle $. Then for any choice of basepoints $p_i \in X_i$, the tuples $(X_i, p_i, G_i, \Gamma _i )$ satisfy the hypotheses of Theorem \ref{thm:main-technical} with the exeption of $X_i$ being simply-connected, and 
        \[          \diam (X_i / \Gamma _i ) \to \infty .                   \]
    \end{Exa}

	\begin{proof}[Proof of Theorem \ref{thm:main-technical}] Assume by contradiction that $\diam (X_i / \Gamma _ i ) \to \infty$. By  Theorem \ref{thm:fukaya-yamaguchi},  after taking a subsequence we can assume the groups $\Gamma_i$ and $G_i$ converge to groups $\Gamma$ and $G$, respectively. Moreover, by \eqref{eq:phi-no-escape} we can assume for each $g \in \mathbb{Z}^n$ the sequence $\varphi_i (g) $ converges to an isometry $\varphi (g) \in \Gamma $, giving us an embedding  $\varphi : \mathbb{Z}^n \to \Gamma$. Denote $\Gamma _ 0 : = \varphi (\mathbb{Z}^n) \leq \Gamma $. 

		 By Theorem \ref{thm:fukaya}, the sequence  $ ( X_i/\Gamma_i , [p_i ]) $ converges in the pointed Gromov--Hausdorff sense to $ ( X/\Gamma , [p] )$,  so $X/ \Gamma$ is not compact. On the other hand, again by Theorem \ref{thm:fukaya}, the sequence $X_i/G_i$ converges in the Gromov--Hausdorff sense to $X/G$, so $\diam (X/G) \leq D$ and $X/G$ is compact. From the equality 
		\[  ( X/ \Gamma ) / (G / \Gamma) = X / G  ,     \]
		we deduce that $G/ \Gamma$ is not compact, and consequently $G/ \Gamma_0$ is not compact either.

		By Lemma \ref{lem:lattice} there is a discrete co-compact subgroup $\overline{H} \leq G / \Gamma _0 $ equipped with an isomorphism $\overline{\psi} : \mathbb{Z}^m \to \overline{H}$ for some $m\geq 1$. Let $H  : = \overline{H} \Gamma_0 \leq G$ be the preimage of $\overline{H}$ in $G$. We can then lift $\overline{\psi} $ to an embedding $\psi : \mathbb{Z}^m \to H$.  Define morphisms $\psi _i : \mathbb{Z}^m \to G_i $ as follows:  for each $j \in \{ 1, \ldots , m \}$, let $\psi _i (e_j)  \in G_i $ be a sequence that converges to $\psi (e_j)$ as $i \to \infty$.  
		
		By putting together $\varphi _i , \varphi $ with $\psi_i, \psi$, we get for $N : = n + m$ morphisms 
		\[
		\Phi_i  :  \mathbb{Z}^N \to G_i, \hspace{2cm}    \Phi  :  \mathbb{Z}^N \to G,
		\]
		with
		\[
		\Phi_i (g,h) : = \varphi_i(g) \psi _i (h) , \hspace{2cm}    \Phi (g,h)  : = \varphi (g) \psi (h) ,
		\]
		for all $g \in \mathbb{Z}^n$, $h \in \mathbb{Z}^m$. By the construction of $\varphi$ and $\psi_i$, we have
		\begin{equation}\label{eq:phis-converge}
			\Phi_i (g) \to \Phi (g)  \text{  for each }g \in \mathbb{Z}^N.    
		\end{equation}
		Notice that while these two maps look very similar,  $\Phi$ is a co-compact embedding with image $H$, but $\Phi_i $ cannot be injective since $N > n$. Denote $H_i : = \Phi _i (\mathbb{Z}^N) \leq G_i$.

		Pick $  r  >  \diam (X/ H)  $ with $d (g p ,p ) \neq 2 r $ for all $g \in H  $, and  define
		\begin{gather*}
			B   =  B_{r} (p ) \subset X , \hspace{2cm}      B_ i   : = B_{r} (p_i) \subset X_i,        \\
			S  : = \{ g \in \mathbb{Z} ^N \, \vert \,  ( \Phi (g) B ) \cap B \neq \emptyset \} .
		\end{gather*}
		By Lemma \ref{lem:generators}, $S$ is a finite generating set of $\mathbb{Z}^N$, so   by Theorem \ref{thm:defining}, there is  $M \in \mathbb{N}$ such that $T : = S^M$ satisfies 
		\[    \{  g \in \mathbb{Z}^N \vert d (\Phi (g) p , p ) \leq 4r \}  \subset T        \]
		and $T^3$ is a defining set of $\mathbb{Z}^N$. 
		\begin{Lem}\label{lem:items} 
			For $i$ large enough, the following holds:
			\begin{enumerate}[label= \textbf{\roman*}]
				\item For all $s \in S$ we have $\Phi_i(s)B_i \cap B_i \neq \emptyset$.\label{item:a1}
				\item If $t \in T^6$ satisfies $\Phi_i (t)B _i \cap B_i  \neq \emptyset$, then $t \in T$. \label{item:a2}
				\item If $h \in H_i $ satisfies $ h B_i \cap B_i \neq \emptyset$, then $  h B_i \subset \Phi _i(T) B_i $. \label{item:a3}
				\item $\Phi_i \vert _{T^6} : T^6 \to H_i$ is injective. \label{item:a5}
			\end{enumerate}
		\end{Lem}
		\begin{proof}
			If \eqref{item:a1} fails, after passing to a subsequence, there would be $s \in S$ with $\Phi_i (s) B_i \cap B_i  = \emptyset $ for all $i$. Pick $x \in \Phi (s) B \cap B$, and $x_i \in X_i$ converging to $x$. Then 
			\[    d( x_i , p_i  ) \to d (x,p)  , \hspace{2cm} d(x_i, \Phi _i (s) p_i) \to d(x , \Phi (s) p ),                                \]
			so 
			\[           x_i  \in \Phi _i (s) B_i  \cap B_i   \text{ for }i \text{  large enough;}  \]
			which is a contradiction.

			If \eqref{item:a2} fails, after passing to a subsequence, there would be $t \in T ^6 \backslash T $ with 
			\[ \Phi _i(t) B_i \cap B_i \neq \emptyset \text{ for all }i. \] 
			Then 
			\[    d ( \Phi (t) p , p   ) = \lim _{i \to \infty} d( \Phi _i (t) p_i, p_i  ) \leq 2r    .       \]
			By our choice or $r$, we have $d ( \Phi (t) p,p ) < 2r $, so $t \in S$; a contradiction.
			
			If \eqref{item:a3} fails, after passing to a subsequence,  there would be $h_i \in H_i$ with $h _i B_i \cap B_i \neq \emptyset$, but $h_iB_i \not\subset \Phi _i (T) B_i$ for all $i$. Pick $x_i \in h_i B_i \backslash \Phi _i (T) B_i$. After passing to a subsequence, we can assume the sequence $x_i$ converges to a point $x \in X$. Pick $t \in \mathbb{Z}^N$ with $d(\Phi (t)p , x ) \leq \diam (X/ H)$. Since 
			\[    d (\Phi (t) p,p) \leq d(\Phi (t) p,x) + d(x,p) \leq \diam (X/ H) + 3r < 4r ,               \]
			we have $t \in T$. Since $d(\Phi_i (t)p_i , x_i ) \to d(\Phi (t) p, x) $, we have 
			\[ x_i \in B_r (\Phi_i (t) p_i) = \Phi _i (t) B_i \subset \Phi_i (T) B_i  \]
			for $i$ large enough; a contradiction.

			If  \eqref{item:a5}  fails, there would be $t \in T^{12}  \backslash \{ 0 \}$ with $\Phi _i (t)  = e \in G_i$ for infinitely many $i$'s. By \eqref{eq:phis-converge}, this would imply $\Phi (t) = e \in H$. However, this contradicts the fact that $\Phi $ is an embedding. 
		\end{proof}
		
		By \eqref{item:a5}, for $i$ large we can identify $T^6$ with a subset of $H_i$, so by Lemma \ref{lem:items},  Theorem \ref{thm:monodromy} applies to the tuple $(X_i, H_i, B_i, S, T)$. Since $T^3$ is defining in $\mathbb{Z}^n$,  the group $\tilde{\Gamma}$ coincides with $\mathbb{Z}^N$ and the natural map $\tilde{\Gamma} \to H_i$ coincides with $\Phi_i$.  Let $\Psi _i : \tilde{X}_i \to X_i$ be the covering map given by Theorem \ref{thm:monodromy}. Since $\Phi _i $ has infinite kernel, $\Psi _i$ is not a homeomorphism. This contradicts the simple-connectedness of $X_i$.    \end{proof}

    \begin{Rem}
        In the above proof, the hypothesis of the spaces $X_i$ being simply-connected was not used until the very end. Thus, if one runs the arguments of the proof using the spaces and groups from Example \ref{ex:cylinder}, one ends up constructing covering spaces $\tilde{X}_i \to X_i$ with $\tilde{X}_i \cong \mathbb{R}^2$ and equipped with faithful $\mathbb{Z}^2$-actions. 
    \end{Rem}
	
	\begin{proof}[Proof of Theorem \ref{thm:separated-subgroup}] By contradiction, assume there are $(X_i,p_i) \in \mathcal{M}$ and discrete groups $G_i \leq \iso (X_i)$ isomorphic to $\mathbb{Z}^n$ with $\diam (X_i / G_i) \leq D$, but for any finite-index subgroup $\Gamma _ i \leq G_i$ with 
		\[      d ( g p_ i, p_i  ) \geq D \text{ for all } g \in \Gamma _i \backslash \{ e \},                 \]
		one has 
		\begin{equation}\label{eq:bad-diameter}
			\diam (X_i  /  \Gamma _i  ) \geq i .          
		\end{equation}
		By  Theorem \ref{thm:gromov}, there are elements $ g_{i,1} , \ldots , g_{i,k_i} \in G_i $ with
		\begin{gather*}
			d (  g p _i , p_i ) \geq D \text{ for all } g \in \langle g_{i,1}, \ldots , g_{i, k_i} \rangle \backslash \{ e \} , \\
			d(  g_{i,j} p_i , p_i   ) \leq 2D \text{ for all } j \in   \{ 1, \ldots , k_i \},               
		\end{gather*}
		and such that  
		\[
		\{ g_{i,1}, \ldots , g_{i,k_i} \}   \text{ is a basis of }G_i^{ab} \otimes \mathbb{R} \cong \mathbb{Z}^n \otimes \mathbb{R} = \mathbb{R}^n.    
		\]
		Consequently, $k_i = n$ for all $i$, and the map $\varphi_i : \mathbb{Z}^n \to \Gamma _ i : =   \langle g_{i,1}, \ldots , g_{i,n } \rangle$ given by 
		\[  \varphi_i (e_j) : = g_{i,j}\text{ for each }j \in \{ 1, \ldots , n \} \] 
		is an isomorphism. Since $\mathcal{M}$ is pre-compact, after passing to a subsequence, we can assume the spaces $(X_i,p_i)$ converge in the pointed Gromov--Hausdorff sense to a space $(X,p)$.  Then \eqref{eq:bad-diameter} contradicts Theorem \ref{thm:main-technical}.
	\end{proof}

	\section{Torus covers} \label{sec:ricci}
	An important feature of closed aspherical Riemannian manifolds is that their universal covers are always non-collapsed \cite[Corollary 3]{guth}.

	\begin{Thm}[Guth]\label{thm:guth}
		Let $M $ be a closed $n$-dimensional Riemannian manifold whose universal cover $\tilde{M}$ is contractible, and equip $\tilde{M}$ with the lifted metric.  Then for each $r > 0 $  there is $p \in \tilde{M}$ with 
		\[         \vol (B_r (p))  \geq \varepsilon (n) r^n .                       \]
	\end{Thm}

	\begin{proof}[Proof of Theorem \ref{thm:torus}] 
	Without loss of generality, we can assume $D \geq 2$.	Let $\mathcal{M}$ be the class of pointed complete simply-connected $n$-dimensional Riemannian manifolds with Ricci curvature $\geq - (n-1)$. It is well known that $\mathcal{M}$ is pre-compact in the pointed Gromov--Hausdorff sense \cite[Section 6]{gromov-polynomial}.  Let $\tilde{M}$ be the universal cover of $M$ equipped with the lifted metric. By Theorem \ref{thm:guth} there is $p \in \tilde{M}$ with 
		\[   \vol(  B_{ 1 }(p)) \geq    \varepsilon (n) .                 \]
		By Theorem \ref{thm:separated-subgroup}, there is a subgroup $\Gamma \leq \pi _1 (M)$ such that 
		\[
		d ( g p , p ) \geq D \text{ for all } g \in \Gamma \backslash \{ e \}  ,
		\]
		and 
        \[ \diam (\tilde{M} / \Gamma ) \leq  D' (n, D ). \] 
        If $M ' : = \tilde{M} / \Gamma $, then $B_{D/2} (p)\subset \tilde{M}$ is sent injectively and locally isometrically to $M'$, so 
		\[ \vol (M') \geq \varepsilon .  \]
	\end{proof}

	\section*{Acknowledgements}
	
       The author is grateful to Cameron Rudd for pointing out the error  in \cite[Lemmas 2.3 and 2.4]{kloeckner-sabourau}  and its connection with Theorem \ref{thm:bns}. The author would also like to thank Christine Escher and Ananda L\'opez-Poo for helpful comments and suggestions, and an anonymous referee for carefully reading the paper and suggesting the inclusion of Example \ref{ex:cylinder}.

\end{document}